\documentclass[11pt, a4paper]{amsart}

\usepackage[english]{babel}
\usepackage[utf8x]{inputenc}
\usepackage[T1]{fontenc}

\usepackage[a4paper,top=3cm,bottom=2cm,left=3cm,right=3cm,marginparwidth=1.75cm]{geometry}

\usepackage{amsmath}
\usepackage{mathtools}
\usepackage{}
\usepackage{graphicx}
\usepackage[colorinlistoftodos]{todonotes}
\usepackage[colorlinks=true, allcolors=blue]{hyperref}
\usepackage{comment}
\usepackage{amsfonts}
\usepackage{amsthm}
\usepackage{newlfont}
\usepackage{amscd}
\usepackage{amsgen}
\usepackage{amssymb}
\usepackage{mathrsfs}
\usepackage{longtable}
\usepackage{listings}
\usepackage{extarrows}
\usepackage{tikz}
\usepackage{tikz-cd}
\usetikzlibrary{matrix,shapes,arrows,decorations.pathmorphing}
\tikzset{commutative diagrams/.cd,
mysymbol/.style={start anchor=center,end anchor=center,draw=none}
}

\theoremstyle{plain} \newtheorem{thm}{Theorem}[section]
 
\newtheorem{lemm}[thm]{Lemma}
\newtheorem{prop}[thm]{Proposition}
\theoremstyle{definition}
\newtheorem{defn}[thm]{Definition}
\newtheorem{rem}[thm]{Remark}

\newtheorem{cor}[thm]{Corollary}

\newtheorem{ex}[thm]{Example}

\newcommand{\Hom}{\text{Hom}}
\newcommand{\git}{\mathbin{
  \mathchoice{/\mkern-6mu/}
    {/\mkern-6mu/}
    {/\mkern-5mu/}
    {/\mkern-5mu/}}}

\title{Torelli problem for Calabi-Yau threefolds with GLSM description}

\author{Micha\l\ Kapustka and Marco Rampazzo}

\begin{document}

\begin{abstract}
\noindent We construct a gauged linear sigma model  with two non-birational K\"alher phases which we prove to be derived equivalent, $\mathbb{L}$-equivalent, deformation equivalent and Hodge equivalent. This provides a new counterexample to the birational Torelli problem which admits a simple GLSM interpretation.
\end{abstract}
\maketitle
\section{Introduction}
\noindent There has been recently growing interest both from a point of view of algebraic geometry and string theory  in the study of derived equivalent but non-isomorphic pairs of Calabi--Yau threefolds. The first and most famous example is the Pfaffian-Grassmannian equivalence observed in \cite{Rodland} and proved in \cite{KuznetsovHPD,BorisovCaldararu}. In this case the equivalence is also interpreted in \cite{HoriTong, AddingtonDonovanSegal} in terms of wall-crossing in the associated gauged linear sigma model (GLSM for short).   This last construction has its roots in physics, in particular quantum field theory: from the seminal paper of Witten \cite{Witten}, a rich literature on the subject emerged, alimented by the profound connection with string theory dualities, in particular mirror symmetry. Different examples have been studied, mainly arising from toric varieties (i.e. giving an abelian GLSM.) while there are still few examples of GLSM associated to non-abelian gauge groups.  Some explicit examples of non-abelian GLSM are studied in \cite{Morrisonetal, DonagiSharpe} and these provide a new insight into mirror symmetry of determinantal Calabi--Yau threefolds. Moreover, a rigorous mathematical description of GLSM has been given in \cite{FanJarvisRuan}.\\
\\
In \cite{SharpePantevetal}, the authors study abelian GLSM theories with two non-birational K\"ahler phases and observe a relation with Kuznetsov's homological projective duality \cite{KuznetsovHPD}. They conjecture, in particular, that two K\"ahler phases of a GLSM are always twisted derived equivalent. Examples of such a phenomenon involving noncommutative varieties as well as partial proofs of the conjecture have been found in \cite{SharpePantevetal,Sharpe}. The case of non-abelian GLSM has been treated from the physics side in \cite{Hori, HoriKnapp, DonagiSharpe} leading to the same conjecture for symmetric and skew-symmetric degeneracy loci. The work of \cite{Hori} has been reinterpreted in mathematical terms in \cite{RennemoSegal}, where the conjecture was proven for both types of degeneracy loci. More generally, a relation between homological projective duality and variations of GIT stability for Landau-Ginzburg models has been established in \cite{Renn17}.\\
\\
Note that up to now most known geometric constructions of non-abelian GLSM admitting two K\"ahler phases and leading to derived equivalent pairs of Calabi--Yau threefolds are obtained by determinantal constructions (see \cite{Hori, HoriKnapp, Renn15, HosonoTakagi, Morrisonetal}).\\
On the other hand, an interesting case of derived equivalent and non-birational pairs of Calabi--Yau threefolds was recently discovered in the context of the Torelli problem. This is the family $\mathcal X_{25}$ of Calabi--Yau threefolds of degree 25 in $\mathbb{P}^9$ studied by \cite{GrossPopescu, Kapustka, Kanazawa, OttemRennemo, BorisovCaldararuPerry}. The elements of this family, first introduced in \cite{GrossPopescu}, are given by the intersection of two generic $PGL(10)$--translates of the Grassmannian $G(2,V_5)$ embedded in $\mathbb P^9$ via the Pl\"ucker embedding.\\
\\
Independently, in \cite{OttemRennemo} and \cite{BorisovCaldararuPerry}, it is proved that for a general Calabi--Yau threefold $\widetilde X\in \mathcal{X}_{25}$ intersecting the projective dual varieties of both translates one obtains another Calabi--Yau threefold $\widetilde Y\in \mathcal{X}_{25}$ which is in general not isomorphic to $\widetilde X$, but which is derived equivalent, deformation equivalent and Hodge equivalent to $\widetilde X$. We shall say in such case that $\widetilde X$ is dual to $\widetilde Y$.  
Such general dual pairs of Calabi--Yau threefolds in $\mathcal{X}_{25}$, in particular, provide counterexamples to the birational Torelli problem. Note that some GLSM interpretation of the duality on $\mathcal{X}_{25}$ has just appeared in \cite{CaldararuKnappSharpe}.\\
\\
In \cite{BorisovCaldararuPerry}, it is additionally shown that the following relation holds in the Grothendieck ring of varieties:
\begin{equation}\label{LequivalenceOttem}
([\widetilde X]-[\widetilde Y]) \mathbb L^4 =0,
\end{equation}
where $\mathbb L$ is the class of the affine line. 
This means that $\widetilde{X}$ and $\widetilde{Y}$ are also so-called $\mathbb L$-equivalent.\\
\\
A similar case has also been discussed in the work of Manivel \cite{Manivel}, where the intersection of two translates of the ten-dimensional spinor variety in the projectivization of a sixteen-dimensional half-spin representation has been investigated. Here the intersection is a Calabi--Yau fivefold, and the projective dual construction gives rise to a non-birational Calabi--Yau fivefold, still, the two varieties have been proven to be deformation equivalent, derived equivalent, $\mathbb L$-equivalent and Hodge equivalent. Moreover, in \cite{Maniveletal}, techniques to construct Calabi--Yau threefolds and fourfolds as orbital degeneracy loci have been explained. This leads, in particular, to all families discussed above and may serve as a source of further examples of derived equivalent and $\mathbb{L}$-equivalent pairs of Calabi--Yau manifolds.\\
\\
Let us point out that the notion of $\mathbb L$-equivalence is somehow related to the notion of derived equivalence. The problem whether derived equivalence may imply $\mathbb L$-equivalence has been first considered in \cite{IMOUnew}, and short after that the positive answer has been stated as a conjecture in \cite{KuznetsovShinder}.
Note that it has already been proven (see \cite{IMOUnew,Efimov}) that there is no implication between derived and $\mathbb L$-equivalence for abelian varieties. Still, up to now,  no counterexample is known among simply connected Calabi--Yau manifolds.\\

\noindent In this paper, we consider the family $\bar{\mathcal{X}}_{25}$ of Calabi--Yau threefolds given as zero loci of sections of the vector bundle $Q^{\vee}(2)$ on $G(2,V_5)$. As it was pointed out in \cite{Kapustka, InoueItoMiura} and \cite{OttemRennemo}, these varieties, still being smooth, belong to the boundary of $\mathcal X_{25}$ and can be interpreted as the intersections of infinitesimal translates of $G(2,V_5)$. For each such a manifold $X$ we provide a construction of a dual Calabi--Yau threefold $Y$ in the same family, which is not birational, but is derived equivalent and $\mathbb L$-equivalent to $X$. Then, as pointed out in \cite[Prop 2.1]{OttemRennemo}, they are also Hodge equivalent i.e. their periods define equivalent integral Hodge structures.
We furthermore observe that our duality concept is an extension of the duality studied in \cite{OttemRennemo,BorisovCaldararuPerry} to the investigated boundary component $\bar{\mathcal{X}}_{25}$ of $\mathcal{X}_{25}$. As explained in \cite{OttemRennemo}, we can then apply the Matsusaka–Mumford theorem (see \cite{MM}), to provide another proof of the fact that a general element of $\mathcal{X}_{25}$ and its dual are not isomorphic.\\
\\
We notice furthermore that for $X,Y\in \bar{\mathcal{X}}_{25}$ a dual pair of Calabi--Yau threefolds we have \begin{equation}\label{our rel of L equivalence}
([X]-[Y]) \mathbb L^2 =0.
\end{equation}
Comparing with (\ref{LequivalenceOttem}), we see that the exponent of $\mathbb L$ in our formula which is valid on the boundary divisor $\bar{\mathcal{X}}_{25}$ is smaller than the exponent known to annihilate the difference of classes of a dual pair in the general case $\mathcal{X}_{25}$.
A similar phenomenon occurs in the Pfaffian--Grassmannian equivalence. The exponent is known only to be bounded by 6 in general and it is proven to be 1 on a boundary divisor (see \cite{IMOU}).
It is an interesting problem proposed in \cite{KuznetsovShinder} to understand the geometric meaning  of the minimal exponent of $\mathbb L$ annihilating a difference of two classes of varieties.
In \cite{KuznetsovShinder} this exponent is conjectured to be related to the ranks of Fourier-Mukai kernels associated to derived equivalences between the two varieties. One of the advantages of our direct approach is that from our proof of derived equivalence of studied pairs of varieties one can explicitly find such Fourier Mukai kernels. Note also that in general it is a nontrivial problem to understand the behaviour of both derived and $\mathbb L$ equivalence in families. We hope that our example, exhibiting  a nontrivial behaviour of these equivalences in families, may provide further insight into that subject. \\
\\
Finally, we present a GLSM description of Calabi--Yau manifolds in our family $\bar{\mathcal{X}}_{25}$ which explains the duality equivalence of $X$ and $Y$ in terms of wall crossing. We thus provide a GLSM construction with two non-birational K\"ahler  phases with simple geometric realizations as zero loci of sections of a vector bundle, which are derived equivalent, deformation equivalent and $\mathbb L$-equivalent. In fact, our GLSM construction is based on a variation of GIT (as in \cite[Sect. 7]{BFK}) and, to our knowledge, it is the only example known so far where such VGIT leads directly to two non-birational Calabi--Yau phases.\\
\\
Our argument relies on the following diagram that we establish in Section \ref{section description of the families}.
\begin{equation}\label{bigdiagram}
\begin{tikzcd}
 & E\arrow[hookrightarrow]{r}\arrow{lddd} & M\arrow{dd}\arrow{lddd}[swap]{f_1}\arrow{rddd}{f_2}\arrow[hookleftarrow]{r} & E'\arrow{rddd} & \\
 & & & & \\
 & & F\arrow{ld}{p} \arrow{rd}[swap]{q} & & \\
X\arrow[hookrightarrow]{r} & G(2,V_5) & & G(3,V_5)\arrow[hookleftarrow]{r} & Y\\
\end{tikzcd}
\end{equation}
The notation is the following:
\begin{itemize}
\item $V_5$ is a five-dimensional vector space and $G(k,V_5)$ stands for the Grassmannian of $k$ dimensional subspaces in $V_5$.
\item $F$ is the flag variety given by the following incidence correspondence:
\begin{equation}
F=\{([V],[W])\in G(2,V_5)\times G(3,V_5) : V\subset W\}
\end{equation}
\item $p$ and $q$ are the natural projections from the flag variety $F$ to the two Grassmannians.
\item  The flag variety $F$ has Picard group generated by the pullbacks of the hyperplane bundles of the two Grassmannians $G(2,V_5)$ and $G(3,V_5)$. We denote the pullbacks of the hyperplane sections of the Grassmannians $G(2,V_5)$ and $G(3,V_5)$ by $\mathcal O(1,0)$ and $\mathcal O(0,1)$ respectively. In this notation $M$ is a hyperplane section of the flag variety $F$, i.e. the zero locus of a section $s\in H^0(F, \mathcal O(1,1))$.

\item We prove (see Lemma \ref{thepushforwards}) that $p_*\mathcal O(1,1)=\mathcal{Q}_1^\vee(2)$ and $q_*\mathcal O(1,1)=\mathcal{U}_2(2)$, where we call $\mathcal U _i$ the universal bundle of a Grassmannian $G(i,V_5)$ and $\mathcal Q_i$ its universal quotient bundle. The varieties $X$ and $Y$ are, respectively, the zero loci of the sections $p_*s$ and $q_*s$ of $\mathcal Q_2^\vee(2)$ and $\mathcal U_3(2)$, 
\item $f_1$ is a fibration over $G(2,V_5)$ with  fiber isomorphic to $\mathbb P^1$, for points outside the subvariety $X$ whereas the fibers are isomorphic to $\mathbb P^2$ for points on $X$. Similarly $f_2$ is a map onto $G(3,V_5)$ whose fibers are $\mathbb P^1$ outside $Y$ and $\mathbb{P}^2$ over $Y$.
\end{itemize}

\noindent In Section \ref{section non birationality}, we  prove that, in general, if $X$ and $Y$ are dual they are not birational. Using the fact that they have Picard number equal to 1, we just need to prove that they are not projectively equivalent. The latter is done in several steps. First, we prove that $X$ and $Y$ are contained in unique Grassmannians. Furthermore the hyperplane section $M$ of $F$ is also uniquely determined both by $X\subset G(2,V_5)$ and by $Y\subset G(3,V_5)$. We then deduce that a linear isomorphism between $X$ and $Y$ must lift to an automorphism of the flag variety $F$ that preserves $M$. We describe explicitly the action of these automorphisms on hyperplane sections of $F$ and find a concrete hyperplane which is not fixed by any of them. \\
\\
The $\mathbb L$-equivalence of $X$ and $Y$ is\ a direct consequence of diagram (\ref{bigdiagram}). It is presented in Section \ref{section L-equivalence}.
In Section \ref{section derived}, we show that the derived categories of coherent sheaves of $X$ and $Y$ can be embedded in two different orthogonal decompositions of the derived category of the hyperplane section $M$ of the flag variety $F$. This fact allows us to prove the derived equivalence of $X$ and $Y$ with a sequence of mutations.
Section \ref{section GLSM} is devoted to the establishing of a GLSM with two K\"ahler phases representing dual Calabi-Yau threefolds from the family $\bar{\mathcal{X}}_{25}$.

\section{The description of the duality}\label{section description of the families}
\noindent Hereafter we will describe the families appearing in diagram (\ref{bigdiagram}) in greater detail. In particular, we define the notion of duality between elements of these families.\\
\\
First of all, with Lemma \ref{thepushforwards} we establish a relation between the vector bundles we described in diagram (\ref{bigdiagram}) proving that the pushforwards of $\mathcal O(1,1)$ are exactly the bundles appearing in the diagram. 
\begin{lemm}\label{thepushforwards}
Let $\mathcal O(1,1)=p^*\mathcal{O}(1)\otimes q^*\mathcal{O}(1) $ be the hyperplane bundle on the flag variety $F$. Then the pushforwards of $\mathcal O(1,1)$ with respect to $p$ and $q$ are, respectively, $\mathcal Q_2^\vee(2)$ and $\mathcal U_3(2)$.
\end{lemm}
\begin{proof}
Observe that the flag variety $F$ can be interpreted as the projectivization of the rank 3 quotient  bundle $\mathcal{Q}^{\vee}$ on the Grassmannian $G(2,V_5)$, hence also the projectivization of $\mathcal{Q}^{\vee}(2)$. In this case, we have the relative Euler sequence
\begin{equation}\label{relativeeuler1}
0\longrightarrow\Omega^1_{G(2,V_5)|F}(a,b)\longrightarrow \mathcal Q_2^\vee(2)\longrightarrow\mathcal O_{F_1}(1)\longrightarrow 0
\end{equation}
and on $G(3,V_5)$
\begin{equation}\label{relativeeuler2}
0\longrightarrow\Omega^1_{G(3,V_5)|F}(a,b)\longrightarrow \mathcal U_3(2)\longrightarrow\mathcal O_{F_2}(1)\longrightarrow 0
\end{equation}
where for $i=2,3$ we called $\mathcal O_{F_i}(1)$ the Grothendieck relative $\mathcal O_{\mathrm{P}(\mathcal E_i)}(1)$ associated to the corresponding bundle $\mathcal E_2=\mathcal Q_2^{\vee}(2)$ and $\mathcal E_3=\mathcal U_3(2)$. We can compute the first Chern class of the relative $\Omega_{G(i,V_5)|F}^1$ from the relative tangent bundle sequences, which is
\begin{equation}
0\longrightarrow T_{G(3,V_5)|F}\longrightarrow T_F\longrightarrow T_{G(3,V_5)}\longrightarrow 0
\end{equation}
and the same sequence holds for $G(2,V_5)$. In both the sequences (\ref{relativeeuler1}) and (\ref{relativeeuler2}), computing the first Chern class we get $\mathcal O_{F_i}(1)=\mathcal O(1,1)$ for $i=1,2$. The remaining part of the proof follows from a general fact that the pushforward of the Grothendieck line bundle of a vector bundle $\mathcal E$ with respect to the surjection to the base is $\mathcal E$.
\end{proof}
\noindent The picture emerging is the following:
\begin{equation}
\begin{tikzcd}
& &\mathcal O(1,1)\arrow{d}\arrow{ld}[swap]{p_*} \arrow{rd}{q_*} & &\\
&  \mathcal Q_2^\vee(2)\arrow{d}& F \arrow{ld}[swap]{p} \arrow{rd}{q} & \mathcal U_3(2)\arrow{d} & \\
X \arrow[hookrightarrow]{r} & G(2,V_5) &  & G(3,V_5)\arrow[hookleftarrow]{r} & Y
\end{tikzcd}
\end{equation}
Moreover, we denote $X:=Z(p_*s)$ the variety of all the points in $G(2,V_5)$ where $p_*s$ vanishes. But $x\in Z(p_*s)$ is equivalent to $s(p^{-1}(x))=0$. Thus, since $F$ is a $\mathbb P^2$ bundle on $G(2,V_5)$, the fibers of the projection from $M$ to $G(2,V_5)$ over points outside $X=Z(p_*s)$ are isomorphic to $\mathbb P^1$, whereas the fibers over $X$ will be isomorphic to $\mathbb P^2$. The same applies to $Y$ in $G(3,V_5)$ and the projection $q|_M$.
\begin{lemm}\label{uniquelydetermined}
Let $X$ be the zero locus of a section $s_2\in H^0(G(2,V_5), \mathcal Q_2^\vee(2))$. Then $s_2$ is uniquely determined by $X$ up to scalar multiplication.\\
Similarly, if $Y$ is the zero locus of a section $s_3$ of $\mathcal U_3(2)$ on $G(3,V_5)$, $s_3$ is uniquely determined by $Y$.
\end{lemm}
\begin{proof}
We will prove the result for $G(2,V_5)$, the proof for the case of $G(3,V_5)$ is identical. Let us suppose $X$ is the zero locus of two sections $s_2$ and $\widetilde{s}_2$. Then, the Koszul resolution with respect to these two sections reads
\begin{equation}
\begin{tikzcd}
\cdots\arrow{r}&\mathcal Q_2(-2)\arrow{r}{\alpha_{s_2}}\arrow{d} & \mathcal I_X\arrow{r}\arrow[equal]{d} & 0 \\
\cdots\arrow{r}&\mathcal Q_2(-2)\arrow{r}{\alpha_{\widetilde{s}_2}} & \mathcal I_X\arrow{r} & 0 
\end{tikzcd}
\end{equation}
thus, in order to have two sections defining the same $X$, the identity of the ideal sheaf, for the stability of $\mathcal Q_2(-2)$,  lifts to an automorphism of $\mathcal Q_2(-2)$. However, since
\begin{equation}
\operatorname{Ext}^\bullet(\mathcal Q_2,\mathcal Q_2)=\mathbb C[0]
\end{equation}
the only possible automorphisms of $\mathcal Q_2(-2)$ are scalar multiples of the identity.
\end{proof}
\begin{cor} Let $X=Z(s_2)\subset G(2,V_5)$. Then there exists a unique hyperplane section $M$ of $F$ such that the fiber $p|_M^{-1}(x)$ is isomorphic to $\mathbb{P}^2$ for $x\in X$ and is isomorphic to $\mathbb{P}^1$  for $x\in G(2,V_5)\setminus X$. Similarly for $Y=Z(s_3)\subset G(3,V_5)$ there exists a unique hyperplane section $M$ of $F$ such that the fiber $q|_M^{-1}(x)$ is isomorphic to $\mathbb{P}^2$ for $x\in Y$ and is isomorphic to $\mathbb{P}^1$  for $x\in G(3,V_5)\setminus Y$.
\end{cor}
\begin{proof} We consider only the case $X=Z(s_2)\subset G(2,5)$ the other being completely analogous. Since $F$ is the projectivization of a vector bundle over $G(2,V_5)$, using Lemma \ref{thepushforwards}, then the pushforward $p_*$ defines a natural isomorphism 
$$H^0(F,\mathcal O(1,1))=H^0(G(2,V_5),\mathcal Q_2^{\vee}(2)).$$
Hence $s_2=p_*(s)$ for a unique $s\in H^0(F,\mathcal O(1,1))$. We define $M=Z(s)$ which satisfies the assertion by the discussion above. The uniqueness of $M$ follows from Lemma \ref{uniquelydetermined}. Indeed, for any hyperplane section $\tilde{M}=Z(\tilde{s})$, the fibers  $p|_{\tilde{M}}^{-1}(x)$ are isomorphic to $\mathbb{P}^2$ exactly for $x\in Z(p_* \tilde{s})$, but  $Z(p_* \tilde{s})=X$ only if $p_* \tilde{s}$ is proportional to $s_2$ which means that $\tilde{s}$ is proportional to $s$ and proves uniqueness.
\end{proof}

Let us consider an isomorphism
\begin{equation}\label{f}
    f:G(2,V_5)\longrightarrow G(3, V_5).
\end{equation}
Every such isomorphism is induced by a linear isomorphism $T_f:V_5\longrightarrow V_5^\vee$ in the following way:\\
\begin{equation}\label{iso}
    f= D\circ\phi_2:G(2,V_5)\longrightarrow G(3, V_5).
\end{equation}
where $D$ is the canonical isomorphism
\begin{equation}
    D: G(i, V_5)\longrightarrow G(5-i, V_5^\vee)
\end{equation}
and $\phi_i$ is the induced action of $T_f$ on the Grassmannian:
\begin{equation}
    \phi_i: G(i, V_5)\longrightarrow G(i, V_5^\vee)
\end{equation}


Note that above maps $f$, $D$, $\phi_2$, $\phi_3$ are restrictions of linear maps between the Pl\"ucker spaces of the corresponding Grassmannians. By abuse of notation we shall use the same name for their linear extensions.
We can now introduce the following notion of duality.
\begin{defn}\label{dualitydefn}
We define two Calabi--Yau threefolds $X\subset G(2,V_5)$ and $Y\subset G(3,V_5)$  to be \emph{dual} to each other if there exists a section $s\in H^0(F,\mathcal O(1,1))$ such that for $s_2=p_*s$ and $s_3=q_*s$
we have $X=Z(s_2)$ and $Y=Z(s_3)$.
\end{defn}
\begin{defn}
Given an isomorphism $f:G(2,V_5)\longrightarrow G(3,V_5)$, we say $X\subset G(2,V_5)$ is $f$-\emph{dual} to $Y\subset G(2,V_5)$ if $f(Y)$ is dual to $X$.
\end{defn}
The following lemmas on duality will be useful in the proof of non-birationality of general dual pairs.\\

Let us start by defining $P=\mathbb P(\wedge^2V_5)\times \mathbb P(\wedge^2V_5^\vee)$, where $\wedge^2 V_5$ is identified with $\wedge^3 V_5^\vee$ by means of $D$. In that case $F$ 
is a proper linear section of codimension 25 of $P$ in its Segre embedding.\\

\begin{rem} \label{eqF}
Recall that the equations of $F$ in $P$ are described by the following sections  $s_{x^*\otimes y}\in H^0(P, \mathcal O (1,1))$ 

\begin{equation}\label{flagequations}
s_{x^*\otimes y}(\alpha, \omega)=\omega(x^*)\wedge\alpha\wedge y
\end{equation}

for $\omega\in\Lambda^2V_5^{\vee}=\Lambda^3V_5$, $\alpha\in\Lambda^2V_5$ and for every $x^*\otimes y\in V_5^{\vee}\otimes V_5$.\\
In other words, we have
$$
s_{x^*\otimes y}(\alpha,\omega)=0 \text{\ for\ }([\alpha], [\omega])\in F(2,3,V_5)\subset \mathbb{P}(\Lambda^2 V_5)\times \mathbb{P}(\Lambda^3 V_5).
$$
\end{rem}

This defines a 25 dimensional subspace $H^0(\mathcal{I}_F(1))\subset H^0(P,\mathcal O(1,1))$ spanned by linearly independent sections corresponding to $x^*=e_i^*$, $y=e_j$ for $i, j\in\{1\dots 5\}$ and a chosen basis $\{e_i\}$ for $V_5$.\\

Now, for every $f$ as in (\ref{f}) we define the following function:
\begin{equation}
\begin{tikzcd}
    P\ni (x,y)\arrow[maps to]{rr}{\iota_f}&&((f^\vee)^{-1}(y),f(x))\in P
\end{tikzcd}
\end{equation}

which induces the following map at the level of sections:
\begin{equation}
\begin{tikzcd}
H^0(P,\mathcal O_P(1,1))\ni s\arrow[maps to]{rr}{\widetilde \iota_f}& & s\circ\iota_f\in H^0(P,\mathcal O_P(1,1)).
\end{tikzcd}
\end{equation}

Note that $\iota_f$ is a linear extension of an automorphism of the flag variety $F\subset P$. It is constructed in such a way that we have that $X$ is defined by a section $p_*(s)\in H^0(G(2,V_5), \mathcal Q_2^\vee(2))$ if and only if $f(X)$ is defined by  $q_*(\widetilde \iota_f (s))\in H^0(G(3,V_5), \mathcal U_3^\vee(2))$. \\

Our aim is to interpret $f$-duality in the setting above as explicitly as possible. For that we will identify $H^0(F, \mathcal O(1,1))$ with a subspace $\mathcal{H}_F$ of sections in $H^0(P, \mathcal O(1,1))$ invariant under our transformations.

\begin{lemm}\label{directsum}
The space $H^0(P, \mathcal O(1,1))$ decomposes as $H^0(I_{F|P}(1,1))\oplus H^0(F, \mathcal O(1,1))$ and the decomposition is invariant under the action of ${\widetilde \iota_f}$ for every isomorphism $f:G(2,V_5)\to G(3,V_5)$. More precisely ${\widetilde \iota_f} H^0(I_{F|P}(1,1))=H^0(I_{F|P}(1,1))$ and there exists a subspace $\mathcal{H}_F \subset H^0(P, \mathcal O(1,1))$ isomorphic to $ H^0(F, \mathcal O(1,1))$ such that ${\widetilde \iota_f}(\mathcal{H}_F)=\mathcal{H}_F$.
\end{lemm} 
\begin{proof}
It is well known that $Aut(F)\simeq  \mathbb{Z}/2\rtimes GL(V_5)$. Moreover, the action of $Aut(F)$ on $F$ is linear and extends to an action of $Aut(F)$ on $P$ compatible with $\widetilde \iota_f$. It follows that $H^0(I_{F|P}(1,1))$ is invariant under ${\widetilde \iota_f}$ since it is clearly invariant under $Aut(F)$. Furthermore the dual action of $Aut(F)$ on $P^{\vee}$ preserves the dual flag variety, hence  $H^0(I_{F^{\vee}|P^{\vee}}(1,1))$ is invariant under the dual action of $\widetilde \iota_f$. We can define $\mathcal{H}_F=H^0(I_{F^{\vee}|P^{\vee}}(1,1))^{\perp}$. The latter space is invariant under $Aut(F)$, so it is also invariant under $\widetilde \iota_f$ and the map 
 $\mathcal{H}_F\to H^0(F, \mathcal O(1,1))$ defined by restriction is an isomorphism. 
\end{proof}
  Note that, by construction, the action of $\widetilde \iota_f$ on $H^0(F, \mathcal O(1,1))$ corresponds to the action $\widetilde \iota_f$ on $\mathcal{H}_F$. It means that we can think of $ H^0(F, \mathcal O(1,1))$  equipped with the action induced by $\widetilde \iota_f$  as a subset of $H^0(P, \mathcal O(1,1))$ invariant under the action of  $\widetilde \iota_f$ on $H^0(P, \mathcal O(1,1))$ .
  
  \begin{rem}\label{eqhf} To get explicit equations defining
   $\mathcal{H}_F$ in terms of matrices
  we apply the procedure from Remark \ref{eqF} to describe the equations of $F^\vee$ with respect to the dual basis of $V_5$. This allows us to have an explicit expression for the equations defining $\mathcal{H}_F$ in $H^0(P, \mathcal O(1,1))$. 
  In our choice of basis we obtain explicit linear conditions on the entries of $10\times 10$ matrices to be elements of $\mathcal{H}_F$.
 \end{rem}
\begin{lemm} \label{fduality} The variety $X$ is $f$-dual to $Y$ if and only if there exists a constant $\lambda\in \mathbb{C}^*$ such that sections $s_X\in \mathcal{H}_F$, $s_Y\in \mathcal{H}_F$ defining $X$ and $Y$ respectively satisfy
$\widetilde{\iota}_f(s_Y)=\lambda s_X $.
\end{lemm}
\begin{proof} By definition, $X$ is $f$-dual to $Y$ if there exists a section $\hat{s}\in H^0(F, \mathcal{O}(1,1))$ such that 
$p_* \hat{s}$ defines $X$ while $q_* \hat{s}$ defines $f(Y)$. By Lemma \ref{directsum} there then exists a unique section $s\in \mathcal{H}_F$ such that $\hat{s}=s|_F$.  Now, by definition of $\widetilde{\iota}_f$, since $q_* s$ defines $f(Y)$  we have $p_* (\widetilde{\iota}_f)^{-1}( s)$ defines $Y$.  Furthermore by Lemma \ref{directsum} we know that $(\widetilde{\iota}_f)^{-1}( s)\in \mathcal{H}_f$. We conclude from Lemmas \ref{thepushforwards}, \ref{uniquelydetermined} and \ref{directsum} that up to multiplication by constants $s=s_X$ and $(\widetilde{\iota}_f)^{-1} (s)=s_Y$. 

\end{proof}

From now on, let us fix a basis of $V_5$ inducing a dual basis on $V_5^{\vee}$, and natural bases on $\wedge^2V_5$ and $\wedge^2V_5^\vee$ which are dual to each other.\\
A section $s\in H^0(P,\mathcal O_P(1,1))$ is represented by a $10\times 10$ matrix $S$ in the following way
\begin{equation}\label{S}
\begin{tikzcd}
s:(x,y)\arrow[maps to]{rr}& & \overline y^T S\overline x
\end{tikzcd}
\end{equation}
where $\overline x$ and $\overline y$ are expansions of $x$ and $y$ in the chosen bases 
of $\wedge^2V_5$ and $\wedge^2V_5^\vee$. 
Once fixed our bases, $\phi_2$ is represented by a $10\times10$ invertible matrix $M_f$, which is the second exterior power of the invertible matrix associated to $T_f$.\\

\noindent We can now describe very explicitly the $f$-duality in terms of matrices using the following. 
\begin{lemm}\label{transposition}

If $S$ is the matrix associated to  $s\in H^0(P,\mathcal O_P(1,1))$ then the matrix associated to  $\widetilde\iota_f(s)$ is $M_f^{-1}S^TM_f$.

\end{lemm}
\begin{proof}
On a pair $(x, y)$, the map $\iota_f$ acts via $\iota_f(x, y)=((\phi_2^\vee)^{-1}(y), \phi_2(x))$. Furthermore, in our choice of basis $\phi_2(x) = M_{f}\overline x$ and $(\phi_2^\vee)^{-1}(y) = (M_{f}^{T})^{-1}\overline y$.\\
This yields:
\begin{equation}
\widetilde\iota_f(s)(x,y)=s\circ\iota_f(x,y) = (M_{f}\overline x)^T S(M_{f}^T)^{-1}\overline y = \overline y^T M_f^{-1}S^T M_f\overline x
\end{equation}.
\end{proof}

\begin{cor}\label{selfdual} A variety  $X=Z(p_*s)\subset G(2,V_5)$ is $f$-dual to itself if and only if the following equality $S^T M_f=S M_f$ holds.
\end{cor}
\begin{proof} By Lemmas \ref{fduality} and \ref{transposition} we have $X$ is $f$-dual to $X$ if and only if there exists a nonzero constant $\lambda$ such that $M_f^{-1}S^T M_f=\lambda S$, but since $S$ and $S^T$ are similar matrices then $\lambda=1$ which ends the proof.
\end{proof}

\begin{rem}
In \cite[sec. 5]{OttemRennemo}, it is proven that $[v]\in\mathbb P(\mathfrak{gl}(V))$ defines a section $s_v$ of $\wedge^2V(1)$, whose projection to $H^0(G(2,V_5),\wedge^2\mathcal Q_2(2))$ cuts out the threefold $X_{[v]}$. Then $s_v$ corresponds to a $10\times 10$ matrix $S$ that we defined in (\ref{S}).  Hence, from Lemmas \ref{fduality} and \ref{transposition} follows that $X_{[v]}$ and $X_{[v^T]}$ are $D$-dual. This means that our duality is equivalent to the duality notion defined in \cite[sec. 5]{OttemRennemo}, extending the  duality defined on $\mathcal X_{25}$.
\end{rem}
\section{Non birationality of dual threefolds}\label{section non birationality}
\noindent In this section, we prove that a general section $s\in H^0(F, \mathcal O(1,1))$ gives rise to two non-isomorphic Calabi--Yau threefolds $X=Z(p_*s)$ and $Y=Z(q_*s)$, this result will be stated in Theorem \ref{thmnonisomorphic}. Before proving the theorem, we will discuss some auxiliary results. In \cite{BorisovCaldararuPerry}, an argument to show that every $\widetilde X\subset\mathcal X_{25}$ is contained in just one pair of Grassmannians has been explained. Using similar ideas, we will prove an analogous result for the boundary $\bar{\mathcal X}_{25}$ of the family, namely that every Calabi--Yau threefold in $\bar{\mathcal X}_{25}$ is contained in just one Grassmannian.\\
\\
From now on we will make extensive use of Borel--Weil--Bott theorem, which allows to compute the cohomology of every Schur functor of the bundles $\mathcal U^\vee$ and $\mathcal Q^\vee$ on a Grassmannian. As we will see, most of the bundles we will deal with can be represented in such a way. For a detailed account on the topic we recommend \cite{BorisovCaldararuPerry}, while for a more general approach on many different formulations of the Borel--Weil--Bott theorem we refer to \cite{Weyman}.
\begin{lemm}\label{restrictioncohomology}
Let $X$ be a Calabi--Yau threefold described as the zero locus of a section of $\mathcal Q_2^\vee(2)$. Then the following equalities hold for every $t\geq 0$:
\begin{equation}\label{cohomologyrestrictiontwist}
H^0(G(2,V_5), \mathcal Q_2(-t)) = H^0(X, \mathcal Q_2|_X(-t));
\end{equation}
\begin{equation}
H^0(G(2,V_5), \wedge^2\mathcal Q_2(-t)) = H^0(X, \wedge^2\mathcal Q_2|_X(-t)).
\end{equation}
In particular, $H^0(X, \mathcal Q_2|_X)\cong V$ and $H^0(X, \mathcal Q_2|_X(-t))=H^0(X, \wedge^2\mathcal Q_2|_X(-t))=0$ for $t$ strictly positive.
\begin{proof}
Let us consider the following short exact sequence which comes from tensoring the ideal sheaf sequence of $X$ with $\mathcal Q_2$:
\begin{equation}\label{idealsheafsequencetensorQ}
0\longrightarrow \mathcal I_{X/G(2,V_5)}\otimes\mathcal Q_2(-t)\longrightarrow\mathcal Q_2(-t)\longrightarrow\mathcal Q_2|_X(-t)\longrightarrow 0
\end{equation}
Given this sequence, we need to show the vanishing of the first two degrees of cohomology for $\mathcal I_{X/G(2,V_5)}\otimes\mathcal Q_2$. To do this, we consider the sequence obtained tensoring with $\mathcal Q_2$ the Koszul resolution of the ideal sheaf of $X$:
\begin{equation}\label{koszulresolutiontensorQ}
0\longrightarrow\mathcal Q_2(-5-t)\xlongrightarrow{\theta}\mathcal Q_2\otimes\mathcal Q_2^\vee (-3-t)\longrightarrow\mathcal Q_2\otimes\mathcal Q_2(-2-t)\xlongrightarrow{\phi}\mathcal I_{X/G(2,V_5)}\otimes\mathcal Q_2(-t)\longrightarrow 0
\end{equation}
The bundles $\mathcal Q_2^\vee(-5-t)$ and $\mathcal Q_2\otimes\mathcal Q_2^\vee(-3-t)$ have no cohomology in degree 0 and 1: this follows from the isomorphisms
\begin{equation}
\mathcal Q_2^\vee(-5-t)\cong \wedge^2\mathcal Q_2^\vee\otimes (\wedge^3\mathcal Q_2^\vee)^{\otimes (4+t)}
\hspace{2pt},
\hspace{10pt}
\mathcal Q_2\otimes\mathcal Q_2^\vee(-3-t)\cong (\wedge^3\mathcal Q_2^\vee)^{\otimes (2+t)}\otimes\wedge^2\mathcal Q_2^\vee\otimes\mathcal Q_2^\vee
\end{equation}
which, in turn, proves $H^0(G(2,V_5),\ker(\phi))=H^0(G(2,V_5),\text{coker}(\theta))=0$ in (\ref{koszulresolutiontensorQ}). Since $\mathcal Q_2\otimes\mathcal Q_2(-2-t)$ has no cohomology in the first two degrees, due to $\mathcal Q_2\otimes\mathcal Q_2(-2-t)\cong(\wedge^2\mathcal Q_2^\vee)^{\otimes (2+t)}$, then also $H^0(G(2,V_5),\mathcal I_{X/G(2,V_5)}\otimes\mathcal Q_2)=0$ and $H^1(G(2,V_5),\mathcal I_{X/G(2,V_5)}\otimes\mathcal Q_2)=0$. This, together with (\ref{idealsheafsequencetensorQ}), proves our claim (\ref{cohomologyrestrictiontwist}). The second equality follows from a totally analogous computation, namely it involves the tensor product of the ideal sheaf sequence with the wedge square of $\mathcal Q_2$.
\end{proof}
\end{lemm}
\begin{lemm}\label{slopestable}
Let $X$ be a Calabi--Yau threefold described as the zero locus of a section of $\mathcal Q_2^\vee(2)$. Then the restriction $\mathcal Q_2^\vee(2)|_X$ is slope-stable.
\end{lemm}
\begin{proof}
The Mumford slope of a vector bundle is invariant up to twists and dualization, so the problem reduces to asking whether $\mathcal Q_2|_X$ is stable. Therefore, let us suppose there exists a subobject $\mathcal F\subset\mathcal Q_2|_X$. Then, since $G(2,V_5)$ has Picard number one, we have $c_1(\mathcal F) = \mathcal O(t)$ and this leads to the injection
\begin{equation}\label{stabilityinjection}
0\longrightarrow \mathcal O \longrightarrow \wedge ^r\mathcal Q_2|_X (-t)
\end{equation}
where $r$ is the rank of $\mathcal F$, which can be one or two. To have $\mathcal F$ as a destabilizing object for $\mathcal Q_2|_X$, $t$ must be strictly positive in order to satisfy the inequality of Mumford slopes
\begin{equation}
\frac{t}{r}=\mu(\mathcal F)\geq \mu(\mathcal Q_2|_X)=\frac{1}{3}.
\end{equation}
On the other hand, the injection in (\ref{stabilityinjection}) means that $\wedge ^r\mathcal Q_2|_X (-t)$ has global sections: what is left to prove is that it can be true only for zero or negative $t$. The latter follows from Lemma \ref{restrictioncohomology}.

\end{proof}
\noindent Let us suppose $X$ is contained in two Grassmannians $G_1$ and $G_2$, where the latter is the image of the former under an isomorphism of $\mathbb P^9$. Since both the restrictions of the normal bundles $\mathcal N_{i}|_X = \mathcal N_{G_i/\mathbb P^9}|_X=\mathcal Q_{2i}^{\vee}(2)|_X$ are stable, every morphism between them must be zero or an isomorphism. Below we furthermore prove that the isomorphism class of the normal bundle determines the Grassmannian. Combining these two facts will give us the uniqueness of the Grassmannian containing $X$.
\begin{lemm}\label{restrictionofsectionsofO1}
Let $X$ be a Calabi--Yau threefold described as the zero locus of a section of $\mathcal Q_2^\vee(2)$. Then the following isomorphism holds:
\begin{equation}
H^0(\mathbb P^9, \mathcal O(1))\cong H^0(X, \mathcal O_X(1))
\end{equation}
\end{lemm}
\begin{proof}
The claim follows by proving separately that $H^0(\mathbb P^9, \mathcal O_{\mathbb P^9}(1))\cong H^0(G(2,V_5), \mathcal O_{G(2,V_5)}(1))$ and $H^0(\mathbb G(2,V_5), \mathcal O_{G(2,V_5)}(1))\cong H^0(X, \mathcal O_X(1))$. The first isomorphism comes from the $\mathcal O(1)$-twist of the Koszul resolution of the ideal sheaf of $G(2,V_5)\subset \mathbb P^9$, which proves the vanishing of the cohomology of $\mathcal I_{G(2,V_5)}(1)$, thus the desired result. The second isomorphism is proved in a similar way, with the resolution
\begin{equation}
0\longrightarrow\mathcal O(-4)\longrightarrow V_5\otimes\mathcal O(-2)\longrightarrow V_5\otimes\mathcal O(-1)\longrightarrow\mathcal I_{G(2,V_5)/\mathbb P^9}(1)\longrightarrow 0
\end{equation}
and Kodaira's vanishing theorem.
\end{proof}
\begin{lemm}\label{normalbundleandgrassmannian}
Let $X\in \bar{\mathcal{X}}_{25}$. Then the isomorphism class of $\mathcal N_{G(2,V_5)/\mathbb P^9}|_X$ determines the isomorphism $\psi:\wedge^2 V_5\to W$, where $V_5$ is a five dimensional vector space and $W\cong H^0(\mathbb P^9, \mathcal O(1))$. 
\end{lemm}
\begin{proof}
Since $\mathcal N_{G(2,V_5)/\mathbb P^9}\cong \mathcal Q_2^\vee(2)$, the restriction of the normal bundle is determined by the restriction of the quotient bundle. Let us begin noting that the surjection
\begin{equation}
V_5\otimes\mathcal O\longrightarrow \mathcal Q_2\longrightarrow 0
\end{equation}
implies the following
\begin{equation}
\wedge^3V_5\otimes\mathcal O\longrightarrow \wedge^3\mathcal Q_2\longrightarrow 0.
\end{equation}
Let us observe that $\wedge^3\mathcal Q_2\cong\mathcal O(1)$. Then, this last surjection tells us that
\begin{equation}
H^0(G(2,V_5), \mathcal O(1))\cong\wedge^3 H^0(G(2,V_5), \mathcal Q_2).
\end{equation}
From Lemma \ref{restrictioncohomology} we have that $H^0(G(2,V_5), \mathcal Q_2)\cong H^0(X, \mathcal Q_2|_X))$, while Lemma \ref{restrictionofsectionsofO1} tells us that $H^0(\mathbb P^9, \mathcal O(1))\cong H^0(X, \mathcal O(1))$. Then, since $H^0(G(2,V_5), \mathcal Q_2)\cong V_5$, we get an isomorphism $\wedge^3 V_5\to W^\vee$ whose dual is exactly $\psi$ since $\wedge^3 V_5\cong \wedge^2 V^\vee_5$.
\end{proof}
\begin{cor} \label{uniquenessGrassmannian} If $X\subset\mathbb{P}^9$ is a Calabi--Yau threefold from the family $\bar{\mathcal{X}}_{25}$, then $X$ is contained in a unique Grassmannian $G(2,5)$ in its Pl\"ucker embedding.
\end{cor}
\begin{proof} Suppose that $X$ is contained in two Grassmannians $G_1$, $G_2$ for each of them we have an exact sequence:
$$0\to \mathcal N_{X|G_i}\to \mathcal N_{X|\mathbb{P}^9}\to \mathcal N_{G_i|\mathbb{P}^9}|_X\to 0 $$
Combining the two exact sequences we obtain a map:
$\phi: \mathcal N_{X|G_1}\to \mathcal N_{G_2|\mathbb{P}^9}|_X$.
Note that we have $\mathcal N_{X|G_i}\simeq \mathcal N_{G_i|\mathbb{P}^9}|_X\simeq \mathcal Q_{2i}^{\vee}(2)|_X$.
By stability of $Q_{2i}^{\vee}(2)$ we have $\phi$ is either trivial or an isomorphism. If it is an isomorphism it induces an isomorphism  $\mathcal N_{G_1|\mathbb{P}^9}|_X\simeq \mathcal N_{G_2|\mathbb{P}^9}|_X$ and we conclude by Lemma \ref{normalbundleandgrassmannian}. If it is trivial it lifts to an isomorphism $\mathcal N_{X|G_1}\simeq \mathcal N_{X|G_2}$ which again gives an isomorphism $\mathcal N_{G_1|\mathbb{P}^9}|_X\simeq \mathcal N_{G_2|\mathbb{P}^9}|_X$ and permits us to conclude again by Lemma \ref{normalbundleandgrassmannian}.

\end{proof}
\noindent Now we are ready to prove the main theorem of this section.
\begin{thm}\label{thmnonisomorphic}
Let $F$ be the partial flag manifold $F(2,3,V_5)$, let $p$ and $q$ be the projections to the two Grassmannians $G(2,V_5)$ and $G(3,V_5)$.\\
Then a general section $s \in H^0(F,\mathcal O(1,1))$ gives rise to two non-birational Calabi--Yau threefolds $X = Z(p_*s)$ and $Y = Z(q_*s)$.
\end{thm}

\begin{proof}

Because of Lemma \ref{uniquelydetermined}, we deduce that if there exists an isomorphism mapping $X$ to $Y$, then it is given by a map $f:G(2,V_5)\to G(3,V_5)$. Recall that such a map is determined by a linear isomorphism from $T_f:V_5\to V_5^\vee$.\\
\\
Thus, because of Corollary \ref{uniquenessGrassmannian}, $X$ and $Y$ are dual and isomorphic only if there exist $f:G(2,V_5)\to G(3,V_5)$ such that $X$ is $f$-dual to $X$.
This, by Corollary \ref{selfdual} translates to the fact that a section $s_X\in \mathcal{H}_F$ from Lemma \ref{directsum} defining $X$ on $F$ is in the fixed locus of the action of $\widetilde\iota_f$ onto $\mathcal{H}_F$. More explicitly this means that $M_f^{-1}S^TM_f=S$ for $S$ being the matrix associated to the section $s_X$. The proof amounts now to check that for $S$ corresponding to an element of $\mathcal{H}_F$ the equation 
$$S^TM-MS=0$$
has no solutions among matrices $M$ of the form $M=\wedge^2 T$. This is done via the following script in Macaulay2 \cite{macaulay2} performed in positive characteristic :
\begin{verbatim}
R=ZZ/17[a_1..a_25]
S=matrix{
{ 1 ,0,0,0,0,0,0,0,0,0},
{0, 0 ,0,0,0,0,0,0,0,0},
{0,0, 0 ,0,0,0,0,0,0,0},
{0,0,0, 0 ,0,0,0,0,0,0},
{0,1,0,0, 0 ,0,0,0,0,0},
{0,0,0,0,0,-1 ,0,0,0,0},
{0,0,0,0,0,0, 1 ,0,0,0},
{0,0,0,0,0,0,0,-1 ,0,0},
{0,0,0,0,0,0,0,0,-1 ,0},
{0,0,0,0,0,0,0,0,0, 1 }}
T=genericMatrix(R,5,5)
M=exteriorPower(2,T)
Sol=ideal flatten(transpose(S)*M-M*S)
saturate(Sol, ideal det T).
\end{verbatim}

Here we chose $S$ a matrix satisfying the equations defining $\mathcal{H}_F= H^0(I_{F^{\vee}|P^{\vee}})^{\perp}$ as in  Remark \ref{eqhf}\\

This implies that a general hyperplane section $s$ of the flag variety $F$ yields two Calabi--Yau threefolds $X$ and $Y$ which are dual, but not  projectively isomorphic. By the fact that the studied manifolds have Picard number one we conclude that they are not birational.
\end{proof}

\begin{rem}
The above proof being very explicit has the advantage that it permits to show a concrete example of a pair of  Calabi--Yau varieties in our family which are dual but not birational. We can however perform a more conceptual proof, which is more susceptible to generalization and permits to estimate the expected codimension of the fixed locus of our duality. It is based on Kleiman transversality of a general translate. We sketch it below. Let us first observe that a general element in $\mathcal{H}_F$ is a matrix with $10$ distinct non-zero eigenvalues. This can be checked in a specific example and expanded by openness. For such elements $S$ the space of matrices $M\in GL(\wedge^2 V)$ which satisfy 
$SM=MS^T$ is a 10 dimensional subset of symmetric matrices. 
To see that, we put $S$ in Jordan normal form $S=J^{-1}DJ$ with $D$ diagonal with distinct nonzero entries and then $J^{-1}DJM=MJ^TDJ^{-T}$ leads to the conclusion that $JMJ^T$ commutes with $D$ hence is diagonal. It follows that $M$ is symmetric and moves in a 10 dimensional family $\mathcal{M}_J$. 

Now note that $GL(10)$ acts transitively on the space of symmetric matrices via $K\cdot M=K^TMK$.
We finally observe that $SM-MS^T=0$ has a solution of the form $M=\wedge^2 N$ with $N\in GL(V)$ exactly when $\mathcal{M}_J\cap \wedge^2 GL(V)$ is non-empty. Observe that the space of symmetric matrices in $\wedge^2 GL(V)$ is of dimension 15 and is represented by elements of $GL(V)$ which are symmetric. We can now perform a dimension count based on Kleiman transversality (\cite[Theorem 2, Lemma 1]{Kleiman}) by finding a map $\theta$ from $GL(10)$ to some variety $B$ whose general fibers are of dimension 25 and meet the locus $$\{G\in GL(10): \exists D \text{ diagonal with distinct nonzero eigenvalues  such that } G^{-1}DG\in H_0\}.$$ The latter fibration exists for dimensional reasons. Then from the inequality $25+10+15<55$ we deduce that for general $b\in B$ we have  $\theta^{-1}(b)\cdot \mathcal{M}_{Id} \cap \wedge^2 GL(V)=\emptyset$. This implies that for every $G\in \theta^{-1}(b)$ and every  $D$ diagonal with distinct nonzero eigenvalues there is no solution to $SM=MS^T$ when $S=G^{-1}DG$ and $M\in \wedge^2 GL(V)$. By our choice of fibration the latter includes some $S\in \mathcal{H}_F$ which completes the proof. 

\end{rem}

\begin{cor} If $\tilde{X}$, $\tilde{Y}$ are general Calabi--Yau threefolds in $\mathcal{X}_{25}$  which are dual in the sense of \cite{OttemRennemo,BorisovCaldararuPerry} then they are not birational. 
\end{cor}
\begin{proof} Consider an open neighborhood $\mathfrak{U}\subset \mathcal{X}_{25}$ of a general $X\in \bar{\mathcal{X}}_{25}$. Consider also the family $\mathfrak{V}$ of duals parametrized by $\mathfrak{U}$. Now $\mathfrak{U}$ and $\mathfrak{V}$ are families of polarized Calabi--Yau threefolds such that, by Theorem \ref{thmnonisomorphic}, there exists a fiber of $\mathfrak U$ which is not isomorphic to the corresponding fiber of $\mathfrak V$. Then by the Matsusaka--Mumford theorem \cite{MM} the corresponding general fibers are not isomorphic and consequently  general dual pairs in $\mathcal{X}_{25}$ are not isomorphic.
\end{proof}

\section{The \texorpdfstring{$\mathbb{L}$}{L}-equivalence in the Grothendieck ring of varieties} \label{section L-equivalence}
\noindent Hereafter we will show how, in the relation (\ref{LequivalenceOttem}), the power of $\mathbb L$ drops to two. This result is due to the characteristics of the fibrations described in (\ref{bigdiagram}), which are special to  $\bar{\mathcal X}_{25}$. 
\begin{thm}\label{thmLeq}
Let $s$ be a generic section of $\mathcal O(1,1)$ on the flag $F$, let $p$ and $q$ be the projections to $G(2,V_5)$ and $G(3,V_5)$. Then, given $X=Z(p_*s)$ and $Y=Z(q_*s)$, we have the following relation in the Grothendieck ring of varieties, where $\mathbb L$ is the class of the affine line.
\begin{equation}
([X] - [Y])\mathbb L^2 = 0
\end{equation}
\end{thm}
\begin{proof}
With the aid of previous results, the proof is immediate from the following claim.\\ 
\\
{\bf Claim.} The maps $\pi_i$ define $\mathbb P^2$-bundles over the Calabi--Yau threefolds and  $\mathbb P^1$-bundles over the complements of the Calabi-Yau threefolds in the Grassmannians. In particular, the maps $\pi_i$ are piecewise trivial fibrations.\\
\\
Indeed, $p^{-1}(X)=\mathbb{P}(Q_2^\vee(2)|_X)$ whereas
$p^{-1}(G(2,V_5)\setminus X)=\mathbb{P}((Q_2(2)|_X)/(p_*s|_{G(2,V_5)\setminus X}) \cdot \mathcal{O}_{G(2,V_5)}|_{G(2,V_5)\setminus X})$. The latter quotient is a vector bundle since $p_*s$ does not vanish outside $X$. The argument for $q$ is completely symmetric.

Using the claim we write the following relations in the Grothendieck ring of varieties:
\begin{equation}
[M]=[X][\mathbb P^2]+[G(2,V_5)\backslash X][\mathbb P^1]
\end{equation}
\begin{equation}
[M]=[Y][\mathbb P^2]+[G(3,V_5)\backslash Y][\mathbb P^1]
\end{equation}
W compare the two expressions and, using properties of the Grothendieck ring of varieties, we get
\begin{equation}
0 = [X]([\mathbb P^2]-[\mathbb P^1])-[Y]([\mathbb P^2]-[\mathbb P^1])
\end{equation}
which, via the formula
\begin{equation}
[\mathbb P^n]=1+\mathbb L + \mathbb L^2 +\dots +\mathbb L^n,
\end{equation}
yields the desired result.
\end{proof}
\section{Derived equivalence}\label{section derived}
\noindent From a theorem of Orlov in \cite{Orlov}, we deduce the following orthogonal decompositions for a hyperplane section of $F$:
\begin{equation}\label{orlov}
\begin{split}
D^b(M) &= \left<D^b(G(2,V_5)), D^b(G(2,V_5))\otimes\mathcal O(1,1), p^*D^b(X)\right>\\
		&= \left<D^b(G(3,V_5)), D^b(G(3,V_5))\otimes\mathcal O(1,1), q^*D^b(Y) \right>
\end{split}
\end{equation}
In the remainder of this section, we will provide a sequence of mutation with the aim of proving the following equivalence of categories:
\begin{equation}
\begin{split}
\left<D^b(G(3,V_5)), D^b(G(3,V_5))\otimes\mathcal O(1,1),  q^*D^b(Y)  \right>  \xrightarrow{\sim}\\ 
\left<D^b(G(2,V_5))  , D^b(G(2,V_5))\otimes\mathcal O(1,1), \Phi D^b(Y)\right>
\end{split}
\end{equation}
where $\Phi$ is a functor given by a composition of mutations. That would prove an equivalence between this last exceptional collection and (\ref{orlov}), thus proving that $D^b(X)\cong D^b(Y)$.\\
\\
Exceptional collections for Grassmannians and flag varieties have been described by Kapranov in \cite{Kapranov}, where a method to construct them has been given in terms of Schur functors of the universal bundle, but we will use the minimal Lefschetz decomposition for $G(2,V_5)$ introduced by Kuznetsov in \cite{KuznetsovIsotropicLines}. The advantage of this collection, which can be recovered from the Kapranov one with a sequence of mutations as explained in \cite{KuznetsovIsotropicLines}, is that it generates a very simple helix involving only twists of two vector bundles. The collection is the following:
\begin{equation}\label{kuznetsov25}
D^b G(2,V_5) = \left<\mathcal O, \mathcal U_2^\vee,\mathcal O(1), \mathcal U_2^\vee (1),\mathcal O(2), \mathcal U_2^\vee (2),\mathcal O(3), \mathcal U_2^\vee (3),\mathcal O(4), \mathcal U_2^\vee (4)\right>
\end{equation}
\noindent The duality isomorphism between $G(2,V_5)$ and $G(3,V_5)$ exchanges $\mathcal U_2^\vee $ with $\mathcal Q_3$ and allows us to write a minimal Lefschetz exceptional collection for $G(3,V_5)$:
\begin{equation}\label{kuznetsov35}
D^b G(3,V_5) = \left<\mathcal O, \mathcal Q_3,\mathcal O(1), \mathcal Q_3(1),\mathcal O(2), \mathcal Q_3(2),\mathcal O(3), \mathcal Q_3(3),\mathcal O(4), \mathcal Q_3(4)\right>.
\end{equation}
Now, before venturing in the computation of the mutations which will lead to the derived equivalence, let us prove some useful cohomology calculations:
\begin{lemm}\label{vanishingQO}
The following relation holds for non negative integers $a, b$ which satisfy $2+a\leq b \leq 7+a$ except for $b = 3+a$:
$$
\operatorname{Ext}^\bullet(\mathcal Q_3(1,b), \mathcal O(2,2+a)) = 0
$$
\end{lemm}
\begin{proof}
The proof is merely an application of Borel--Weil--Bott theorem, in particular, we are interested in understanding on which conditions on $a$ and $b$ we can obtain $H^0(F, \mathcal Q_3^\vee (1,2+a-b))=0$.\\
\\
Due to the Leray spectral sequence, our problem simplifies to showing that the pushforward of this bundle with respect to one of the two projections from the flag has no cohomology.\\
Namely, due to the projection formula, we have:
\begin{align*}
q_*\mathcal Q_3^\vee (1,2+a-b)  &= \mathcal U_3(1)\otimes\mathcal Q_3^\vee (2+a-b)=\wedge^2\mathcal U_3^\vee\otimes\mathcal Q_3^\vee (2+a-b) \\
	&= \wedge^2\mathcal U_3\otimes \left(\wedge^3\mathcal U_3\right)^{\otimes(2+a)}\otimes\mathcal Q_3^\vee\otimes\left(\wedge^2\mathcal Q_3^\vee\right)^{\otimes b}
\end{align*}
The Bott-Weil theorem states that cohomology vanishes in every degree if two or more of the following integers coincide:
$$
9+a;\hspace{5pt} 8+a;\hspace{5pt} 5+a;\hspace{5pt} 3+b;\hspace{5pt} 1+b.
$$
and this completes the proof.
\end{proof}
\noindent A similar result can be obtained with the same argument:

\begin{lemm}\label{vanishingOO}
The following relation holds for non negative integers $a, b$ which satisfy $3+a\leq b \leq 7+a$:
$$
\operatorname{Ext}^\bullet(\mathcal O(1,b), \mathcal O(2,2+a)) = 0
$$
\end{lemm}
\noindent Another useful vanishing condition comes from the Leray spectral sequence:
\begin{lemm}\label{vanishingpushforward}
Let $\mathcal F$ and $\mathcal F'$ be vector bundles on $F$ such that they are pullbacks of vector bundles on the same Grassmannian. Then the following relation holds for every $a$, $b$, $c$, $d$ such that $d-b$ is either one or two:
\begin{equation}
\operatorname{Ext}^\bullet(\mathcal F(a,b), \mathcal F'(c,d)) = 0\\
\end{equation}
\end{lemm}
\begin{proof}
We observe that
\begin{align*}
\mathcal F^\vee(-a,-b)\otimes\mathcal F'(c,d) = p^*(\mathcal F^\vee\otimes\mathcal F'(c-a))\otimes q^*\mathcal O(d-b).
\end{align*}
The pushforwards of $q^*\mathcal O(-1)$ and $q^*\mathcal O(-2)$ to $G(2,V_5)$ have no cohomology, thus $\mathcal F^\vee(-a,-b)\otimes\mathcal F'(c,d)$ is acyclic. Due to the Leray spectral sequence we have
\begin{equation}
H^0(F,\mathcal F^\vee(-a,-b)\otimes\mathcal F'(c,d))=H^0(G(2,V_5),p_*\mathcal F^\vee(-a,-b)\otimes\mathcal F'(c,d))
\end{equation}
and this yields the desired result.
\end{proof}
\noindent The following lemmas provide some useful mutations which we will use in the further computations.
\begin{lemm}\label{mutationUQ}
We have the following mutation in the derived category of a Grassmannian $G(k,V_5)$ for every choice of the integers $a,b$:
$$L_{\mathcal O(a,b)}\mathcal U(a,b) =\mathcal Q(a,b)$$
\end{lemm}
\begin{proof}
The following fact
\begin{equation}
\operatorname{Ext}^\bullet(\mathcal Q(a,b), \mathcal O(a,b)) = \mathbb C^n[0]
\end{equation}
follows from Borel--Weil--Bott theorem, it tells us that the mutation we are interested in is the cone of the morphism
\begin{equation}
V_5\otimes\mathcal O\longrightarrow \mathcal Q.
\end{equation}
From the universal sequence
\begin{equation}
0\longrightarrow\mathcal U\longrightarrow V_5\otimes\mathcal O\longrightarrow\mathcal Q\longrightarrow 0
\end{equation}
we see that the morphism is surjective, thus the cone yields the kernel $\mathcal U$.
\end{proof}
\begin{lemm}\label{getridoftilde}
In the derived category of $G(3,V_5)$ the following mutations can be performed:
\begin{align}
R_{\mathcal O(a+1, b-1)}\mathcal Q_3(a,b) = \mathcal Q_2(a,b)\\
R_{\mathcal O(a+1, b-1)}\mathcal U_3(a,b) = \mathcal U_2(a,b)
\end{align}
\end{lemm}
\begin{proof}
With Borel--Weil--Bott theorem we can compute the following:
\begin{equation}
\operatorname{Ext}^\bullet(\mathcal Q_3(a,b), \mathcal O(a+1,b-1)) = \mathbb C[-1]
\end{equation}
so a mutation involving that Ext is an extension. The relevant exact sequence is
\begin{equation}\label{sequenceOQQ}
0\longrightarrow\mathcal O(1,-1)\longrightarrow\mathcal Q_2\longrightarrow\mathcal Q_3\longrightarrow 0,
\end{equation}
which can be found computing the rank one cokernel of the injection $\mathcal U_2\xhookrightarrow{}\mathcal U_3$, comparing the universal sequences of the two Grassmannians and applying the Snake Lemma, this proves our first claim.\\
In order to verify the second one, we write the sequence involving the injection between the universal bundles, which is
\begin{equation}\label{sequenceUUO}
0\longrightarrow\mathcal U_2\longrightarrow\mathcal U_3\longrightarrow\mathcal O(1,-1)\longrightarrow 0.
\end{equation}
The related Ext, in this case, is $\mathbb C[0]$, so the mutation is the cone of the relevant morphism, yielding the desired result.
\end{proof}
\noindent Now we are ready to introduce the following result, which is the key of the proof of the derived equivalence.
\begin{prop}\label{orthogonal}
Let $X$ and $Y$ the zero loci of the pushforwards of $s\in H^0(F,\mathcal O(1,1))$. Then the following functor is an equivalence of categories
\begin{equation}
\begin{split}
\left<D^b(G(3,V_5)), D^b(G(3,V_5))\otimes\mathcal O(1,1),  q^*D^b(Y)  \right>  \xrightarrow{\sim}\\ 
\left<D^b(G(2,V_5))  , D^b(G(2,V_5))\otimes\mathcal O(1,1), \Phi\circ q^* D^b(Y)\right>
\end{split}
\end{equation}
where $\Phi$ is a functor given by a composition of mutations.
\end{prop}
\begin{proof}
The idea of the proof is writing the collection for the hyperplane section in a way such that we can use our cohomology vanishing results to transport line bundles $\mathcal O(a+1,b-1)$ to the immediate right of $\mathcal Q_3(a,b)$, then use Lemma \ref{getridoftilde} to get rid of $\mathcal Q_3(a,b)$, thus transforming vector bundles on $G(2,V_5)$ to vector bundles on $G(3,V_5)$.
The exceptional collection for $M$ with the $G(3,V_5)$ description is the following:
\begin{align*}
&D^b(M)=\\
&\left<\mathcal O, \mathcal Q_3,\mathcal O(0,1), \mathcal Q_3(0,1),\mathcal O(0,2), \mathcal Q_3(0,2),\mathcal O(0,3), \mathcal Q_3(0,3),\mathcal O(0,4), \mathcal Q_3(0,4),\right.\\
&\left.
\mathcal O(1,1), \mathcal Q_3(1,1),\mathcal O(1,2), \mathcal Q_3(1,2),\mathcal O(1,3),  \mathcal Q_3(1,3),\mathcal O(1,4), \mathcal Q_3(1,4),\mathcal O(1,5), \mathcal Q_3(1,5),q^*D^b Y\right>
\end{align*}
Our first operation is sending the first five bundles to the end, they get twisted  with the anticanonical bundle of $M$, which, with the adjunction formula, can be shown to be $\omega_M=\mathcal O(2,2)$.
\begin{align*}
&D^b(M)=\\
&\left<\mathcal Q_3(0,2),\mathcal O(0,3), \mathcal Q_3(0,3),\mathcal O(0,4), \mathcal Q_3(0,4), \mathcal O(1,1), \mathcal Q_3(1,1),\mathcal O(1,2), \mathcal Q_3(1,2),\mathcal O(1,3),
\right.\\&\left.
\mathcal Q_3(1,3),\mathcal O(1,4), \mathcal Q_3(1,4),\mathcal O(1,5), \mathcal Q_3(1,5),\mathcal O(2,2), \mathcal Q_3(2,2),\mathcal O(2,3), \mathcal Q_3(2,3),\mathcal O(2,4), \phi_1 D^b Y\right>
\end{align*}
where we introduced the functor
\begin{equation}
\phi_1 = R_{\left<\mathcal O(2,2), \mathcal Q_3(2,2),\mathcal O(2,3), \mathcal Q_3(2,3),\mathcal O(2,4)\right>}\circ q^*
\end{equation}
Applying Lemma \ref{vanishingQO}, we observe that $\mathcal O(1,1)$  can be moved next to $\mathcal Q_3(0,2)$. Then we can use Lemma \ref{getridoftilde} sending $\mathcal Q_3(0,2)$ to $\mathcal Q(0,2)$. This can be done twice due to the invariance of the operation up to overall twists, yielding:
\begin{align*}
&D^b(M)=\\
&\left<\mathcal O(1,1),\mathcal Q_2(0,2),\mathcal O(0,3), \mathcal Q_3(0,3),\mathcal O(0,4), \mathcal Q_3(0,4), \mathcal Q_3(1,1),\mathcal O(1,2), \mathcal Q_3(1,2),\mathcal O(1,3),
\right.\\&\left.
\mathcal O(2,2),\mathcal Q_2(1,3),\mathcal O(1,4), \mathcal O(2,3),\mathcal Q_3(1,4),\mathcal O(1,5), \mathcal Q_3(1,5), \mathcal Q_3(2,2), \mathcal Q_3(2,3),\mathcal O(2,4), \phi_1 D^b Y\right>
\end{align*}
We are tempted to perform the same operation with $\mathcal Q_3(0,3)$ and $\mathcal O(1,2)$, but $\mathcal O(1,2)$ cannot pass through the bundles in between, since there are non-vanishing $\mathrm{Ext}$ involved. We can avoid the problem using the fact that $\mathcal Q_3(1,1)\cong \mathcal Q_3^\vee(1,2)$ and that we can mutate this last bundle in $\mathcal U_3^\vee(1,2)$ acting with $\mathcal O(1,2)$, due to the dual formulation of Lemma \ref{mutationUQ}.\\
Again, all these operations can be performed twice:
\begin{align*}
&D^b(M)=\\
&\left<\mathcal O(1,1),\mathcal Q_2(0,2),\mathcal O(0,3),\mathcal O(1,2),\mathcal Q(0,3),\mathcal O(0,4), \mathcal Q_3(0,4), \mathcal U_3^\vee(1,2), \mathcal Q_3(1,2),\mathcal O(1,3),
\right.\\&\left.
\mathcal O(2,2),\mathcal Q_2(1,3),\mathcal O(1,4), \mathcal O(2,3),\mathcal Q(1,4),\mathcal O(1,5), \mathcal Q_3(1,5), \mathcal U_3^\vee(2,3), \mathcal Q_3(2,3),\mathcal O(2,4), \phi_1 D^b Y\right>.
\end{align*}
Now, $\mathcal O(0,3)$ and $\mathcal U_3^\vee(1,2)$ qualify for a mutation of the type described in Lemma \ref{getridoftilde}, to get them closer to each other we observe that, due to Lemma \ref{vanishingOO}, the Ext between $\mathcal O(0,3)$ and $\mathcal O(1,2)$ vanishes, and, for a similar application of Borel--Weil--Bott theorem, also the Exts between $\mathcal U_3^\vee(1,2)$ and the two bundles at its left are zero. Applying the same sequence of mutations to the $(1,1)$-twist of these objects we get the following collection:
\begin{align*}
&D^b(M)=\\
&\left<\mathcal O(1,1),\mathcal Q_2(0,2),\mathcal O(1,2),\mathcal U_2(0,3),\mathcal U_2^\vee(1,2),\mathcal O(0,3),\mathcal O(0,4), \mathcal Q_3(0,4),  \mathcal Q_3(1,2),\mathcal O(1,3),
\right.\\&\left.
\mathcal O(2,2),\mathcal Q_2(1,3),\mathcal O(2,3), \mathcal U_2(1,4),\mathcal U_2^\vee(2,3),\mathcal O(1,4),\mathcal O(1,5), \mathcal Q_3(1,5),  \mathcal Q_3(2,3),\mathcal O(2,4), \phi_1 D^b Y\right>.
\end{align*}
Again, thanks to the dual formulation of Lemma \ref{mutationUQ}, $\mathcal O(1,3)$ can mutate $\mathcal Q_3(1,2)$ to $\mathcal U_3^\vee(1,3)$, so we can apply Lemma \ref{getridoftilde} to transform $\mathcal Q_3(0,4)$ in $\mathcal Q_2(0,4)$. But then $\mathcal O(1,3)$ ends up next to $\mathcal O(0,4)$, which is orthogonal to it by application of Lemma \ref{vanishingOO}, so they can be exchanged. Passing through $\mathcal Q_2(0,4)$ via Lemma \ref{mutationUQ} and mutating it to $\mathcal U_2(0,4)$, $\mathcal O(0,4)$ goes right next to $\mathcal U_3^\vee(1,3)$, which is mutated to $\mathcal U_2^\vee(1,3)$ by applying Lemma \ref{getridoftilde}.\\
Once we have done the same for the $(1,1)$-twists, we have transformed all the rank 2 and rank 3 vector bundles on $G(3,V_5)$ in vector bundles on $G(2,V_5)$.
What we still need to do is to remove the twists involving powers of the hyperplane class of $G(3,V_5)$ and, consequently, recognize an exceptional collection of $G(2,V_5)$ and its twist.
Removing all the duals we get the following result:
\begin{align*}
&D^b(M)=\\
&\left<\mathcal O(1,1),\mathcal Q_2(0,2),\mathcal O(1,2),\mathcal U_2(0,3),\mathcal U_2(2,2),\mathcal O(0,3),\mathcal O(1,3),\mathcal U_2(0,4),\mathcal U_2(2,3),\mathcal O(0,4),
\right.\\&\left.
\mathcal O(2,2),\mathcal Q_2(1,3),\mathcal O(2,3), \mathcal U_2(1,4),\mathcal U_2(3,3),\mathcal O(1,4),\mathcal O(2,4),\mathcal U_2(1,5),\mathcal U_2(3,4),\mathcal O(1,5),  \phi_1 D^b Y\right>.
\end{align*}
First we send $\mathcal O(1,1)$ to the end, then we use Lemma \ref{vanishingpushforward} to order the bundles by their power of the second twist:
\begin{align*}
&D^b(M)=\\
&\left<\mathcal Q_2(0,2),\mathcal O(1,2),\mathcal U_2(2,2),\mathcal O(2,2),\mathcal U_2(0,3),\mathcal O(0,3),\mathcal O(1,3),\mathcal U_2(2,3),\mathcal Q_2(1,3),\mathcal O(2,3),
\right.\\&\left.
\mathcal U_2(3,3),\mathcal O(3,3)\mathcal U_2(0,4),\mathcal O(0,4),\mathcal U_2(1,4),\mathcal O(1,4),\mathcal O(2,4),\mathcal U_2(3,4),\mathcal U_2(1,5),\mathcal O(1,5), \phi_2 D^b Y\right>
\end{align*}
where we defined
\begin{equation}
\phi_2 = R_{\mathcal O(3,3)}\circ\phi_1.
\end{equation}
Now we send the last 10 objects to the beginning and reorder again the collection with respect to the second twist, obtaining the following:
\begin{align*}
&D^b(M)=\\
&\left<\mathcal U_2(1,1),\mathcal O(1,1),\mathcal U_2(-2,2),\mathcal O(-2,2),\mathcal U_2(-1,2),\mathcal O(-1,2),\mathcal O(0,2),\mathcal U_2(1,2),\mathcal Q_2(0,2),\mathcal O(1,2),
\right.\\&\left.
\mathcal U_2(2,2),\mathcal O(2,2)\mathcal U_2(-1,3),\mathcal O(-1,3),\mathcal U_2(0,3),\mathcal O(0,3),\mathcal O(1,3),\mathcal U_2(2,3),\mathcal Q_2(1,3),\mathcal O(2,3), \phi_3 D^b Y\right>,
\end{align*}
where
\begin{equation}
\phi_3 = L_{\left<\mathcal U_2(3,3),\mathcal O(3,3),\mathcal U_2(0,4),\mathcal O(0,4),\mathcal U_2(1,4),\mathcal O(1,4),\mathcal O(2,4),\mathcal U_2(3,4),\mathcal U_2(1,5),\mathcal O(1,5)\right>}\circ\phi_2
\end{equation}
Now we observe that $\mathcal Q_2(0,2)$ is orthogonal to $\mathcal U_2(1,2)$, so they can be exchanged: this allows us to mutate $\mathcal Q_2(0,2)$ to $\mathcal U_2(0,2)$ sending it one step to the left. After doing the same thing with $\mathcal O(1,1)$--twists of these bundles, the last steps are tensoring everything with $\mathcal O(-2,-2)$ and sending the first two bundles to the end.\\
We get:
\begin{align*}
&D^b(M)=\\
&\left<\mathcal U_2(-4,0),\mathcal O(-4,0),\mathcal U_2(-3,0),\mathcal O(-3,0),\mathcal U_2(-2,0),\mathcal O(-2,0),\mathcal U_2(-1,0),\mathcal O(-1,0),\mathcal U_2(0,0),\mathcal O(0,0),
\right.\\&\left.
\mathcal U_2(-3,1),\mathcal O(-3,1),\mathcal U_2(-2,1),\mathcal O(-2,1),\mathcal U_2(-1,1),\mathcal O(-1,1),\mathcal U_2(0,1),\mathcal O(0,1), \mathcal U_2(1,1),\mathcal O(1,1),\phi_4 D^b Y\right>
\end{align*}
We defined the last functor
\begin{equation}
\phi_4 = \mathcal T(-2,-2)\circ R_{\left<\mathcal U_2(1,1),\mathcal O(1,1)\right>}\circ\phi_3
\end{equation}
where $\mathcal T(-2,-2)$ is the twist with $\mathcal O(-2,-2)$.\\
Now, if we observe the first half of the collection, we can recognize $D^bG(2,V_5)$: in fact, if we take the Kuznetsov collection (\ref{kuznetsov25}), we can transform $\mathcal U_2$ to $\mathcal U_2^\vee (-1)$ in every Lefschetz block. Then, acting repeatedly with the canonical bundle to send object from the end to the beginning of the collection, we get our result, once we define $\Phi\circ q^* = \phi_4$. 
\end{proof}
\noindent We have shown that both $D^b(X)$ and $\phi_4 D^b(Y)$ can figure as the last block of the first row in (\ref{orlov}), so, for the uniqueness of the orthogonal complement, there is an equivalence of categories
\begin{equation}
D^b(X)\rightarrow\phi_4 D^b(Y)
\end{equation}
Moreover, it is a known fact that the left and the right mutations define an action of the braid group on the set of exceptional collections: the right mutation provides an inverse for the left mutation, as explained, for example, in \cite{HoriIqbalVafa} and \cite{ShinderNotes}. Thus we deduce that the categories $D^b(Y)$ and $\phi_4 D^b(Y)$ are equivalent.\\
\\
Summing all up, the content of this section provides a proof for the following theorem:
\begin{thm}\label{thmDeq}
Let $X$ and $Y$ be dual Calabi--Yau threefolds in the sense of Definition \ref{dualitydefn}. Then they are derived equivalent.
\end{thm}
\section{The GLSM construction}\label{section GLSM}
\noindent In this final section we will give a GLSM realization of dual pairs $(X,Y)$ of Calabi--Yau threefolds in $\bar {\mathcal{X}}_{25}$. Namely, we will construct a gauged linear sigma model with two Calabi--Yau phases associated to different chambers of the space of the stability parameter such that the critical loci are dual threefolds $Y$ and $X$.\\
The mathematical description of the GLSM we will use throughout this work is due to Okonek, to whom we are very thankful for his insights, while a thorough exposition of the subject has been given by Fan, Jarvis and Ruan in \cite{FanJarvisRuan}. In their work, as an example, a similar construction of the Grassmannian $G(k,n)$ as a GIT quotient with respect to $GL(k,\mathbb C)$ has been constructed, and the GLSM of a section of $\bigoplus_j\mathcal O(d_j)$ has been investigated, giving a formal definition of the critical loci in both the phases appearing in the model.
\begin{defn}
Let $V$ be a vector space endowed with the action of a reductive group $G$. We call \textit{gauged linear sigma model} the data of a $G$-invariant function
\begin{equation}
\begin{tikzcd}
V\arrow{r}{w} & \mathbb C.
\end{tikzcd}
\end{equation}
called \emph{superpotential}. Furthermore, we define \emph{critical locus} associated to the superpotential $w$ the following variety:
\begin{equation}
\operatorname{Crit}(w) = Z(dw).
\end{equation}
Fixed a character
\begin{equation}
\begin{tikzcd}
G\arrow{r}{\rho} & \mathbb C^* \\
g\arrow[mapsto]{r} &  \rho g
\end{tikzcd}
\end{equation}
the notion of semistability
\begin{equation}
V^{ss}_\rho=\{v\in V : \overline{\{(\rho^{-1} g, gv)|g\in G\}}\cap\{0\}\times V=\emptyset\}
\end{equation}
allows us to define the \emph{vacuum manifold} as a GIT quotient:
\begin{equation}
\mathcal M_\rho=\operatorname{Crit}(w)\git_\rho G.
\end{equation}
\end{defn}
\noindent The notion of phase transition is encoded in the variation of stability conditions: namely, changing the character $\rho$ leads to different vacuum manifolds. According to the theory of stability conditions, the regions of the space of characters characterized by the same GIT quotients are called \textit{chambers}, thus the problem of phase transitions of a GLSM is interpreted as a problem of wall crossing.
\begin{ex}
Let $G$ be a reductive group, $\mathcal E = \mathcal P\times_G \mathcal F$ a vector bundle with base $\mathcal B  = \mathcal P/G$ and $s\in H^0(\mathcal B, \mathcal E)$. Given a $G$-module $\mathcal U$ containing $\mathcal P$ with $\text{cod}_{\mathcal U}(\mathcal U\backslash \mathcal P)\geq 2$, we define the \textit{gauged linear sigma model} of $s$ to be a map
\begin{equation}
\begin{tikzcd}
\mathcal U\times\mathcal F^\vee\arrow{r}{\check{s}} & \mathbb C \\
(u,\lambda)\arrow[mapsto]{r} &  \lambda\cdot\hat{s}(u)
\end{tikzcd}
\end{equation}
where the function
\begin{equation}
\begin{tikzcd}
\hat{s}:\mathcal U\arrow{r} & \mathcal F
\end{tikzcd}
\end{equation}
is completely defined by $s$, namely by the requirement of satisfying the $G$-equivariancy condition
\begin{equation}
s([p])=[p,\hat{s}(p)].
\end{equation}
and, then, extended uniquely to $\mathcal U$, which is always possible as long as the above condition on the codimension is fulfilled.
\end{ex}
\noindent By asumption there exists a character $\rho$ whose $G$-semistability condition on $\mathcal U\times\mathcal F^\vee$ has semistable locus  $V^{ss}_{\rho_0}=\mathcal P\times\mathcal F^\vee$. In this case, the critical locus of the superpotential will be determined by the following. 
\begin{lemm}[Okonek's lemma]\label{okonekslemma}
Let $\check{s}$ be a superpotential defined by a regular section $s\in H^0(\mathcal B, \mathcal E)$. Then the following isomorphism holds:
\begin{equation}
\operatorname{Crit}(\check{s}) \cong Z(\hat s). 
\end{equation}
\end{lemm}
\begin{proof}
By definition, we have
\begin{equation}
Z(d\check{s})=\{(u,\lambda)\in\mathcal E^\vee : \hat s(u)=0, \lambda\cdot d\hat s(u)=0 \}.
\end{equation}
Since $s$ is a regular section, then $\hat s$ is regular. Then, since its Jacobian $d\hat s$ has maximal rank, $\lambda\cdot d\hat s(u)=0$ if and only if $\lambda = 0$.
\end{proof}
\noindent Then the vacuum manifold will be the GIT quotient of the zero locus of $\hat s$ with respect to $G$. This, in turn, gives
\begin{equation}
\mathcal M_{\rho_0} = Z(s).
\end{equation}
We observe that this construction can be used to realize the zero locus of a section of a homogeneous vector bundle as a phase of a GLSM, provided a family of characters such that the GIT quotient with respect to a given chamber yields the right subset of the vector space $\mathcal U\times\mathcal F^\vee$.\\
\\
\noindent Varying the character $\rho$ leads to different semistable loci, which, in turn, define different GIT quotients. These are called \emph{phases} of the physical theory. An interesting physical problem is to discuss \textit{phase transitions} of a gauged linear sigma model, which means wall-crossing between different chambers.\\
\\
In the following, we will present our GLSM construction leading to the varieties discussed above. First, we will give the following characterization of the bundle $\mathcal U_3(2)$ over $G(3,V_5)$:
\begin{equation}
\begin{tikzcd}
&\hspace{-15pt}\mathcal U_3(2)=\frac{\Hom(\mathbb C^3, V_5)\backslash\{\operatorname{rk}<3\}\times\mathbb C^3}{GL(3,\mathbb C)}\arrow{dd}&\hspace{-20pt}\ni&\hspace{-20pt}(B,v)\sim(Bg^{-1}, \det g^{-2}g v)\\ 
&&&\\
&\hspace{-45pt}G(3,V_5)=\frac{\Hom(\mathbb C^3, V_5)\backslash\{\operatorname{rk}<3\}}{GL(3,\mathbb C)}\arrow[bend right=30,swap]{uu}{s}&\hspace{-20pt}\ni&\hspace{-60pt}B\sim Bg^{-1}.
\end{tikzcd}
\end{equation}
In this setting, chosen a rank three $5\times 3$ matrix $B$, the section $s$ is defined by the following:
\begin{equation}\label{thesection}
s(B) \hspace{5pt}=\hspace{5pt} \left(B, \hspace{5pt}\hat s_1(B) b_1+ s_2(B) b_2+ s_3(B) b_3\right),
\end{equation}
where $b_i$ are the three columns of $B$. Thus, in order to respect the expected degree, $\hat s$ must be a vector of three quintics in the entries of $B$. In this way we have defined the image of $s$ in $V_5\otimes\mathcal O(2)$. In particular, since $\mathcal U_3(2) = q_*\mathcal O(1,1)$, the quintics $\hat s_i(B)$ will be such that the second coordinate in (\ref{thesection}) will be a vector of five polynomials which are quadratic in the $3\times 3$ minors of $B$. 
Moreover, we see that $s$ extends to a map
\begin{equation}
\begin{tikzcd}[row sep=small]
\Hom(\mathbb C^3, V_5)\arrow{rr}{s}&&\Hom(\mathbb C^3, V_5)\times\mathbb C^3\\
 B\arrow[maps to, shorten <= 2.3em, shorten >= 2.2em]{rr} && (B, \hat s(B)).
\end{tikzcd}
\end{equation}
From the definition of $\hat s$ we construct the following superpotential:
\begin{equation}
\begin{tikzcd}[row sep=small]
\Hom(\mathbb C^3, V_5)\times(\mathbb C^3)^\vee\arrow{rr}{s}&&\mathbb C\\
 B,\omega\arrow[maps to, shorten <=3.2em, shorten >= -2.1em]{rr} &&\hspace{25pt} \omega\cdot\hat s(B)
\end{tikzcd}
\end{equation}
Note that this formulation of a GLSM fits into the physical description of \cite{HoriTong}. In particular, the choice of a superpotential of the form given by \cite[(2.6)]{HoriTong} can be written, in physical terms, as
\begin{equation}
W = \int d^2\theta \operatorname{Tr}(PB\hat{s}(B)),
\end{equation}
where $\omega = PB$ and $P_1,\dots P_5$ are superfields transforming as $P\mapsto\det g^2 P$ under the gauge group, which is $U(3)$, and the integration is on two fermionic coordinates of the superspace.\\
\\
Now, let $\rho_\tau$ be the character defined by $\rho_\tau(g)=\det g^{-\tau}$. This leads to two different chambers in the space of stability conditions.
\subsubsection{The chamber $\tau>0$}
A pair $(B,\omega)$ is stable if there are no sequences $\{g_n\}$ satisfying \begin{equation}
\lim_{n\to\infty}\det g_n = 0
\end{equation}
such that the sequence $\{(B g_n^{-1},\det (g_n)^2\omega g_n^{-1})\}$ has a limit. We observe that the term $B g_n^{-1}$ will always diverge in the limit, unless $B$ has not maximal rank. In this latter case, it will be possible to choose a sequence $g_n$ such that $g_n^{-1}$ has no limit, but $B g_n^{-1}$ is finite. Since $\det (g_n)^2\omega g_n^{-1}$ is always finite, we get no further condition on $\omega$.\\
\\
Thus the GIT quotient relative to the chamber $\tau>0$ will define the bundle $\mathcal U_3^\vee(-2)$ over $G(3,V_5)$ and the vacuum manifold, due to Lemma \ref{okonekslemma}, is isomorphic to the Calabi--Yau threefold $Y=Z(s)$.
Moreover, being the superpotential $\check s$ $G$-invariant,the map $\check s_+$ in Diagram (\ref{s+}) is well defined:
\begin{equation}\label{s+}
\begin{tikzcd}
\mathcal U_3^\vee(-2)\arrow{rr}{\check s_+}\arrow{dd} & & \mathbb C &&\mathcal U_3(2)\arrow{dd}\\
&&&&\\
\hspace{-25pt}Z(s)\subset G(3,V_5)&&&& G(3,V_5).\arrow[bend right= 30,swap]{uu}{s}
\end{tikzcd}
\end{equation} 
\subsubsection{The chamber $\tau<0$}
Here, in order to achieve semistability, we need to test our pairs $(B,\omega)$ with sequences $g_n$, where $\det g_n$ tends to infinity. In this setting, we claim that the semistable locus is given by the following set:
\begin{equation}\label{semistablelocus-}
V^{ss}_- = \{(B,\omega)\in \Hom(\mathbb C^3, V_5)\times (\mathbb C^\vee)^3 : \omega \neq 0, \ker\omega\cap\ker B = 0 \}.
\end{equation}
First of all, the case $\omega=0$ is ruled out by the fact that there always exist a sequence $\{g_n\}$ with $\det g_n \to\infty$ such that $(Bg_n^{-1},0)$ has a limit. Thus, let us suppose $\omega\neq 0$.
To show that the set described in (\ref{semistablelocus-}) contains the semistable locus, let us suppose $\ker\omega\cap\ker B$ is non trivial.
Then we fix a basis of $V_5$ and $\mathbb{C}^3$, where
\begin{equation}
B =\left(\begin{matrix}
0 & b_{12} & b_{13}\\
0 & b_{22} & b_{23}\\
0 & b_{32} & b_{33}\\
0 & b_{42} & b_{43}\\
0 & b_{52} & b_{53}\\
\end{matrix}\right);\hspace{5pt}\omega =\left(\begin{matrix}
0 & \omega_{2} & \omega_{3}\\
\end{matrix}\right).
\end{equation}
We can then exhibit a sequence $\{g_n\}$, with $\det g_n=n$, such that both $\omega$ and $B$ are fixed under its action. This is achieved, for example, with
\begin{equation}
g_n^{-1} =\left(\begin{matrix}
n^3 & 0 & 0 \\
0 & 1/n^2 & 0\\
0 & 0 & 1/n^2
\end{matrix}\right).
\end{equation}
To prove the other inclusion, we must show that, if $\ker\omega\cap\ker B=0$, there is no sequence $\{g_n\}$ with $\det g_n \to\infty$ such that the sequence  $g_n\cdot (B,\omega)$ has a limit.\\
Again, we can fix a basis of $V_5$ in order to achieve
\begin{equation}
B =\left(\begin{matrix}
1 & 0 & 0\\
0 & 1 & 0\\
0 & 0 & b_{33}\\
0 & 0 & 0\\
0 & 0 & 0\\
\end{matrix}\right);\hspace{5pt}\omega =\left(\begin{matrix}
0 & 0 & 1\\
\end{matrix}\right).
\end{equation}
In that case we define $(B_n,\omega_n)$ the pair $g_n\cdot (B,\omega)$. Then we form $M_{n}$ to be the $3\times 3$ matrix whose first two rows are the first two rows of $B_n$ and the third row is $\omega_n$. Then we note that
\begin{equation}
\det M_n =\det g_n^2 \det g_n^{-1}=\det g_n\to \infty
\end{equation}
hence $M_n$ has no limit, so neither does $g_n \cdot (B,\omega)$.\\
\\
We observe that, since $\ker\omega$ is two-dimensional, the condition $\ker\omega\cap\ker B=0$ implies $\operatorname{rk}B\geq 2$, otherwise the kernels would intersect in a non trivial vector space.\\
\\
The critical locus of our superpotential, in the phase $\tau<0$, is described by the following equations in $V^{ss}_-$:
\begin{equation}
Z(d\check{s})=\left\{\begin{matrix}\omega\cdot d\hat{s} = 0 \\ \hat{s} = 0\end{matrix}\right.
\end{equation}
The request of having $\omega\neq 0$ in the kernel of the transpose of $d\hat{s}$ can be rephrased saying that the Jacobian of $\hat{s}$ has a non-trivial kernel and this is not possible if $B$ is maximal rank. This fact, combined with the condition $\operatorname{rk}B\geq 2$, yields $\operatorname{rk}B=2$, which automatically satisfies $\hat{s}=0$.\\
\\
In the following we will determine the explicit expression for the functions $\hat s(B)$ via the pushforward of the general expression of a hyperplane section of the flag. This  determines uniquely a section of $\mathcal U_3(2)$ on $G(3,V_5)$ and we can read $\hat s(B)$ by confronting the result with (\ref{thesection}).
We will adopt the convention of the summation of repeated indices in order to lighten the notation. Furthermore the square brackets encasing a set of indices will mean that a tensor is made antisymmetric with respect to permutation of those indices, namely
\begin{equation*}
T_{[i_1,\dots i_k]}=\frac{1}{k!}\sum_{\sigma\in S_k}\epsilon_\sigma T_{\sigma(i_1)\dots\sigma(i_k)}
\end{equation*}
where $\epsilon_\sigma$ is the sign of the permutation $\sigma$.\\
\\
A general section $S\in H^0(F,\mathcal O(1,1))$ can be written in the following way:
\begin{equation}
S(A,B) =S^{ijklm}\psi_{ijk}(B)\psi_{lm}(A),
\end{equation}
where $A$ is the matrix given by a basis of a representative of a point in $G(2,V_5)$, while $B$ is the same for $G(3,V_5)$, $\psi_{lm}(A)$ is the $2\times 2$ minor of $A$ obtained choosing the rows $l$ and $m$ and $\psi_{ijk}(B)$ is the $3\times 3$ minor of $B$ defined in the same way. Note that the functions $\psi$ are, by definition, completely antisymmetric, thus S will be antisymmetric with respect to $(i,j,k)$ and $(l,m)$.\\
Let us choose a basis of $V_5$ such that $A$ is given by the second and the third columns of $B$. Thus we can use the linearity of $\psi_{klm}(B)$ with respect to the variables $b_{r1}$ and write $S$ in the following ways:
\begin{equation}\label{S(A)}
S(A,B) = S^{ijklm}\psi_{[ij}(A)b_{k]1}\psi_{lm}(A);
\end{equation}
\begin{equation}\label{S(B)}
S(A,B) = S^{ijklm}\psi_{ijk}(B)\frac{\partial}{\partial{b_{p1}}}\psi_{plm}(B)
\end{equation}
From (\ref{S(A)}) we can write the pushforward $s_1$ of $S$ to $G(2,V_5)$: seeing $(b_{11},\dots,b_{51})$ as a vector in $\mathcal Q_{[A]}$, the usual inner product in $V_5$ allows us to define, as an element of $\mathcal Q_2^\vee(2)$, the vector whose $r$-th component is
\begin{equation}
\begin{split}
s_{1,r} &= S^{ijklm}\frac{\partial}{\partial b_{r_1}}\psi_{[ij}(A)b_{k]1}\psi_{lm}(A) \\
&= S^{ijklm}\psi_{[ij}(A)\delta_{k]r}\psi_{lm}(A)
\end{split}
\end{equation}
In a similar way, we can define from (\ref{S(B)}) a section $s_2$ of $\mathcal U_3(2)$, if we note that $\{\partial_{b_{11}},\dots ,\partial_{b_{51}}\}$ define a basis of linear functionals on $\mathcal U_{3[B]}$. We get:
\begin{equation}\label{s2(B)}
s_2=S^{ijklm}\psi_{ijk}(B)\left(\begin{matrix}
		\psi_{1lm}(B)\\
		\psi_{2lm}(B)\\
		\psi_{3lm}(B)\\
        \psi_{4lm}(B)\\
        \psi_{5lm}(B)
\end{matrix}\right).
\end{equation}
In the description of the GLSM, we defined a section of $\mathcal U(2)$ with the following expression:
\begin{equation}\label{thesectionagain}
s_2(B) = B,\hspace{5pt} \hat{s}_1(B)\left(\begin{matrix}
		b_{11}\\
		b_{21}\\
		b_{31}\\
        b_{41}\\
        b_{51}
\end{matrix}\right)
+\hat{s}_2(B)\left(\begin{matrix}
		b_{12}\\
		b_{22}\\
		b_{32}\\
        b_{42}\\
        b_{52}
\end{matrix}\right)
+\hat{s}_3(B)\left(\begin{matrix}
		b_{13}\\
		b_{23}\\
		b_{33}\\
        b_{43}\\
        b_{53}
\end{matrix}\right).
\end{equation}
Confronting the last equation with (\ref{s2(B)}) leads us to write the following expression for $\hat s(B)$:
\begin{equation}\label{shat}
\hat{s}_r(B) = S^{ijklm}\psi_{ijk}(B)\frac{\partial}{\partial b_{pr}}\psi_{plm}(B).
\end{equation}
In the above, we wrote $\hat s$ as a function defined on $\Hom(\mathbb C^3,V_5)\backslash\{\operatorname{rk}<3\}$ with values in $\mathbb C^3$, but we note that, as expected, it extends by zeros to a function on all $\Hom(\mathbb C^3,V_5)$. Namely, if the rank of $B$ is smaller than three, all the $3\times3$ minors vanish, so $\hat s(B)=0$.
Then, by inspection, we see that $\hat{s}_i$ is linear in the entries of the $i$-th column of $B$ and  quadratic in the entries of the other two columns.\\
\\
Now, since $\operatorname{rk} B=2$, let us choose a basis where the first column of $B$ vanishes. This reduces the system of 15 equations $\omega\cdot d\hat{s} = 0$ to five quartics. The overall factor $\omega_1$ appearing in each of them can be discarded since the choice of having $b_1=0$ and the condition $\ker B\cap \ker\omega=0$ imply $\omega_1\neq 0$. Moreover, since the five quartics are independent on the entries of $b_1$, they are quadrics with respect to the $2\times 2$ minors of the matrix obtained discarding the first column from $B$. Summing all up, the critical locus for the phase $\tau<0$ is given by
\begin{equation}
\operatorname{Crit}(\check s) = \{(B,\omega): \ker B\cap\ker\omega=0; \operatorname{rk}B=2, \partial_{b_{1i}}\hat s_1=0\}.
\end{equation}
Finally, computing the derivatives of (\ref{shat}) with respect to the entries of the first column of $B$, we get
\begin{equation}
\frac{\partial}{\partial b_{p1}}\hat{s}_1(B) = S^{ijklm}\psi_{[ij}(A)\delta_{k]q}\psi_{lm}(A)
\end{equation}
which are exactly the quadrics appearing in (\ref{S(A)}).\\
\\
\noindent So far, we got no conditions on $\omega$ except for $\omega_1\neq 0$: the  critical locus of the superpotential in the chamber $\tau<0$ is a bundle $\mathcal E$ over the zero locus of the five quadrics in $G(2,V_5).$
However, we still have a $GL(3,\mathbb C)$-action on this bundle: a matrix $B$ with zeros in the first column is fixed by a stabilizer of $GL(3,\mathbb C)$ given by matrices of the form
\begin{equation}
g_n^{-1} =\left(\begin{matrix}
a & b & c \\
0 & 1 & 0\\
0 & 0 & 1
\end{matrix}\right)
\end{equation}
with $a\neq 0$ and all the triples $(\omega_1,\omega_2,\omega_3)$ with nonvanishing $\omega_1$ lie in the same orbit with respect to this stabilizer, which acts freely on them. So the action is transitive and free. Quotienting $\mathcal E$ with respect to the $GL(3,\mathbb C)$-action, yields exactly the Calabi--Yau threefold $X$, this proves the compatibility of our GLSM construction with the description of diagram (\ref{bigdiagram}).
\subsection*{Acknowledgements}
We would like to thank Atsushi Ito, John Christian Ottem and Paul Francis de Medeiros for clarifying discussions, and Shinnosuke Okawa for  showing us the reference \cite{IMOUnew}. We are also grateful to the referee for pointing out an error in a previous version of the paper.
M. Rampazzo was supported by the PhD program at the University of Stavanger.  M. Kapustka was supported by the project NCN 2013/10/E/ST1/00688. 

\vskip10pt
\author{Micha\l{} Kapustka:}
\address{University of  Stavanger, Department of Mathematics and Natural Sciences, NO-4036 Stavanger, Norway \\
Institute of Mathematics of the Polish academy of Sciences ul. Śniadeckich 8, 00-656 Warszawa, Poland.}
\email{michal.kapustka@uis.no}
\\
\\
\author{Marco Rampazzo:}
\address{University of  Stavanger, Department of Mathematics and Natural Sciences, NO-4036 Stavanger, Norway.}
\email{marco.rampazzo@uis.no}
\end{document}